\documentclass[journal]{IEEEtran}

\usepackage{extarrows,paralist}
\usepackage{mhchem} 
\usepackage{siunitx} 
\usepackage{graphicx} 
\usepackage{amsmath,graphicx}
\usepackage{mathtools}

\usepackage{amsthm}
\newtheorem{lem}{Lemma}
\newtheorem{theorem}{Theorem}

\newtheorem{assump}{Assumption}

\usepackage[square, comma, sort&compress, numbers]{natbib}
\usepackage{mathtools}
\usepackage{graphicx}
\usepackage{float}
\usepackage{caption}
\usepackage{color}
\usepackage{placeins}
\usepackage{float}
\usepackage{tabularx,colortbl}
\usepackage[skip=2pt,font=footnotesize]{caption}
\usepackage{amsmath,subfigure}
\usepackage{listings}
\usepackage[outdir=./]{epstopdf}
\usepackage{amssymb}
\usepackage{multirow}
\usepackage{algorithm}
\usepackage{algpseudocode}

\usepackage{enumitem}
\allowdisplaybreaks

\makeatletter
\newcommand{\vast}{\bBigg@{4}}
\newcommand{\Vast}{\bBigg@{5}}
\newcommand{\vertiii}[1]{{\left\vert\kern-0.25ex\left\vert\kern-0.25ex\left\vert #1 
    \right\vert\kern-0.25ex\right\vert\kern-0.25ex\right\vert}}
\makeatother

\def\mbb{\mathbb}
\def\mb{\mathbf}
\def\mc{\mathcal}

\def\bs{\boldsymbol}

\begin{document}
	\title{\huge Linear convergence in optimization over directed graphs\\ with row-stochastic matrices} 
	\author{Chenguang Xi$^\dagger$, Van Sy Mai$^\ddagger$, Ran Xin$^\dagger$, Eyad H. Abed$^\ddagger$, and Usman A. Khan$^\dagger$
		\thanks{
			$^\dagger$C.~Xi, R. Xin, and U.~A.~Khan are with the Department of Electrical and Computer Engineering, Tufts University, 161 College Ave, Medford, MA 02155; {\texttt{chenguang.xi@tufts.edu, ran.xin@tufts.edu, khan@ece.tufts.edu}}. This work has been partially supported by an NSF Career Award \# CCF-1350264.
			$^\ddagger$V. S. Mai and E. H. Abed are with the Department of Electrical and Computer Engineering and the Institute for Systems Research, University of Maryland, College Park, MD 20742; \texttt{\{vsmai,abed\}@umd.edu}.}
	}
	\maketitle
	\begin{abstract}
		This paper considers a distributed optimization problem over a multi-agent network, in which the objective function is a sum of individual cost functions at the agents. We focus on the case when communication between the agents is described by a \emph{directed} graph. Existing distributed optimization algorithms for directed graphs require at least the knowledge of the neighbors' out-degree at each agent (due to the requirement of column-stochastic matrices). In contrast, our algorithm requires no such knowledge. Moreover, the proposed algorithm achieves the best known rate of convergence for this class of problems, $O(\mu^k)$ for $0<\mu<1$, where $k$ is the number of iterations, given that the objective functions are strongly-convex and have Lipschitz-continuous gradients.  Numerical experiments are also provided to illustrate the theoretical findings.
	\end{abstract}

\section{Introduction}\label{s1}
In distributed optimization problems, a network of agents cooperatively minimizes a sum of local cost functions. Formally, we consider a decision variable $\mb{x}\in\mbb{R}^p$ and a strongly-connected network containing $n$ agents where each agent $i$ has only access to a local objective function $f_i(\mb{x}):\mbb{R}^p\rightarrow\mbb{R}$. The agents aim to minimize the sum of their objectives, $\sum_{i=1}^nf_i(\mb{x})$, through local information exchange. This class of problems has recently received much attention and has found various applications, e.g., in distributed learning,~\cite{distributed_Boyd,ml,ml2}, source localization,~\cite{distributed_Rabbit,distributed_Khan}, machine-learned physics,~\cite{li2016pure,li2016understanding,brockherde2017bypassing}, and formation control \cite{formationcontrol,formationcontrol2}. 

Many distributed optimization methods have been developed in recent years. The initial approaches were based on gradient descent,~\cite{uc_Nedic,cc_Lobel2,cc_nedic,cc_Ram2}, which is intuitive and computationally simple but usually slow due to the diminishing step-size used in the algorithm. \begin{color}{black}The convergence rates are shown to be~$O(\frac{\ln k}{\sqrt{k}})$ for arbitrary convex functions and~$O(\frac{\ln k}{k})$ for strongly-convex functions.\end{color} Afterwards, methods based on Lagrangian dual variables were developed, which include distributed dual decomposition,~\cite{dual_Terelius}, and distributed implementation of Alternating Direction Method of Multipliers (ADMM),~\cite{ADMM_Mota,ADMM_Wei,ADMM_Shi}. The convergence rates of these approaches accelerate to~$O(\mu^k),0<\mu<1$, for strongly-convex functions albeit at the expense of higher computation. To achieve both fast convergence rate and computational simplicity, some distributed algorithms do not (explicitly) use dual variables while keeping  a constant step-size. \begin{color}{black}For example, the distributed Nesterov-based method,~\cite{fast_Gradient}, achieves~$O(\frac{\ln k}{k^2})$ for arbitrary convex function under the bounded and Lipschitz gradients assumptions. \end{color}It can be interpreted to have an inner loop, where information is exchanged, within every outer loop where the optimization-step is performed. Refs.~\cite{EXTRA,Augmented_EXTRA} use a constant step-size and the history of gradient information to achieve an exact convergence  \begin{color}{black} at~$O(\frac{1}{k})$ for general convex functions and at $O(\mu^k)$ for strongly-convex functions.\end{color} 

All these methods,~\cite{uc_Nedic,cc_Lobel2,cc_nedic,cc_Ram2,dual_Terelius,ADMM_Mota,ADMM_Wei, ADMM_Shi,fast_Gradient,EXTRA,Augmented_EXTRA}, assume the multi-agent network to be an undirected (or balanced) graph, i.e., if agent~$i$ sends information to agent~$j$, then agent~$j$ also sends information to agent~$i$. In practice, however, it may not always be possible to assume undirected (or balanced) communication. For example, agents may broadcast at different power levels implying communication capability in one direction but not in the other. Moreover, an algorithm may still remain convergent even after removing a few slow communication link; a procedure that may result in a directed graph. It is of interest, thus, to develop optimization algorithms that are fast and are applicable to directed graphs. The challenge, however, arises in the fact that doubly-stochastic matrices, required typically in distributed optimization, can not be obtained, in general, over directed graphs. As a result, one is restricted to only a row-stochastic or a column-stochastic weight matrix. 

We now report the related work on optimization over directed graphs. The (Sub)gradient-Push (SP),~\cite{opdirect_Nedic,opdirect_Tsianous,opdirect_Tsianous2,opdirect_Tsianous3}, employs only column-stochastic matrices by applying the push-sum consensus,~\cite{ac_directed,ac_directed0}, to distributed (sub)gradient descent,~\cite{uc_Nedic}. Inspired by these ideas, Directed-Distributed (Sub)gradient Descent (D-DSD),~\cite{D-DGD,D-DPS}, combines surplus consensus,~\cite{ac_Cai1}, and distributed (sub)gradient descent,~\cite{uc_Nedic}. Ref.~\cite{opdirect_Makhdoumi} combines the weight-balancing technique in~\cite{c_Hooi-Tong} with the distributed (sub)gradient descent,~\cite{uc_Nedic}. These (sub)gradient-based methods,~\cite{opdirect_Nedic,opdirect_Tsianous,opdirect_Tsianous2,opdirect_Tsianous3,D-DGD,D-DPS,opdirect_Makhdoumi}, restricted by diminishing step-sizes, converge at~$O(\frac{\ln k}{\sqrt{k}})$. To accelerate the convergence, DEXTRA,~\cite{DEXTRA} combines push-sum consensus,~\cite{ac_directed,ac_directed0}, and EXTRA,~\cite{EXTRA}, and converges linearly under the strong-convexity assumption with the step-size being chosen in some interval. This interval of step-size is later relaxed in Refs.~\cite{ADD-OPT,opdirect_nedicLinear}, while keeping the linear convergence. Note that although constructing a doubly-stochastic matrix is avoided, all of the aforementioned methods,~\cite{opdirect_Nedic,opdirect_Tsianous,opdirect_Tsianous2,opdirect_Tsianous3,D-DGD,D-DPS,opdirect_Makhdoumi,DEXTRA,ADD-OPT,opdirect_nedicLinear}, require each agent to know its out-degree to implement a column-stochastic weight matrix. This requirement is impractical in many situations, especially when the agents use a broadcast-based communication.

\begin{color}{black}Compared with column-stochastic matrices, row-stochastic matrices are much easier to implement in a distributed fashion as each agent can locally decide the weights. The proposed algorithm in this paper builds on Ref.~\cite{opdirect_Mai}, which employs only row-stochastic matrices with the convergence rate  of~$O(\frac{\ln k}{\sqrt{k}})$ for arbitrary convex functions. In this paper, we propose a fast and fully distributed algorithm to solve the distributed optimization problem over directed graphs that only requires row-stochastic matrices. To overcome the imbalance\footnote{\begin{color}{black}Recall from consensus that with row-stochastic matrices, all agents agree on some (linear-convex) function of the initial conditions whose coefficients are not necessarily~$\frac{1}{n}$, thus creating an imbalance in the way the initial conditions are weighted. In distributed optimization, this leads to convergence to a suboptimal solution, please see~\cite{D-DGD} for precise convergence arguments.\end{color}} that is caused by employing only row-stochastic matrices, we have at each agent an additional variable that converges asymptotically to the left eigenvector of the row-stochastic weight matrix. The gradient information is later divided by this additional variable to cancel the imbalance. For the protocol to work with only row-stochastic matrices, we do however need to assume that agents have and know their identifiers, e.g.,~$1,\cdots,n$. The algorithm requires no information of agents' out-degree and achieves a linear convergence rate, i.e.,~$O(\mu^k),0<\mu<1$, for strongly-convex functions.\end{color}

The remainder of the paper is organized as follows. Section~\ref{s2} formulates the problem and describes the algorithm with appropriate assumptions. Section~\ref{s3} states the main convergence results, the proofs of which are provided in Section~\ref{s4}. We show numerical results in Section~\ref{s5} and Section~\ref{s6} concludes the paper.

\textbf{Notation:} We use lowercase italic letters to denote scalars, lowercase bold letters to denote vectors, and uppercase italic letters to denote matrices, e.g.,~$x\in\mbb{R}$,~$\mb{x}\in\mbb{R}^n$,~$X\in\mbb{R}^{n\times n}$, for some positive integer~$n$. The matrix,~$I_n$, represents the~$n\times n$ identity, and~$\mb{1}_n$ and~$\mb{0}_n$ are the~$n$-dimensional vectors of all~$1$'s and~$0$'s. We denote~$\mb{e}_i=[0,\cdots,1_i,\cdots,0]^\top$. For an arbitrary vector~$\mb{x}$, we denote~$[\mb{x}]_i$ as its~$i$th element. For a differentiable function~$f(\mb{x}):\mbb{R}^p\rightarrow\mbb{R}$,~$\nabla\mb{f}(\mb{x})$ denotes the gradient of~$f$ at~$\mb{x}$. The spectral radius of a matrix,~$A$, is represented by~$\rho(A)$ and~$\lambda_i(A)$ denotes the~$i$th largest (in magnitude) eigenvalue of~$A$. For a matrix~$X$, we denote~$\mbox{diag}(X)$ to be a diagonal matrix consisting of the corresponding diagonal elements of~$X$ and zeros everywhere else. \begin{color}{black}For an irreducible, row-stochastic matrix,~$A$, we denote its right and left eigenvectors corresponding to the eigenvalue of~$1$ by~$\mb{1}_n$ and~$\bs{\pi}^\top$, respectively, such that~$\bs{\pi}^\top\mb{1}_n = 1$. Depending on its argument, we denote~$\|\cdot\|$ as either a particular matrix norm, the choice of which will be clear in Lemma~\ref{lem3}, or a vector norm that is compatible with this particular matrix norm, i.e.,~$\|A\mb{x}\|\leq\|A\|\|\mb{x}\|$ for all matrices,~$A$, and all vectors,~$\mb{x}$. The notation~$\|\cdot\|_2$ denotes two-norm of vectors and matrices. Since all vector norms on finite-dimensional vector space are equivalent, we have the following:~$c{'}\|\cdot\|\leq\|\cdot\|_2\leq c\|\cdot\|, d{'}\|\cdot\|_2\leq\|\cdot\|\leq d\|\cdot\|_2$, where~$c{'},c,d{'},d$ are some positive constants. See~\cite{hornjohnson:13} for details on vector and matrix norms. \end{color}

\section{Problem, Assumptions, and Algorithm}\label{s2}
In this section, we formulate the distributed optimization problem, formalize the assumptions, and describe our algorithm. Consider a strongly-connected network of~$n$ agents communicating over a \emph{directed} graph,~$\mc{G}=(\mc{V},\mc{E})$, where~$\mc{V}$ is the set of agents and~$\mc{E}$ is the collection of ordered pairs,~$(i,j),i,j\in\mc{V}$, such that agent~$j$ can send information to agent~$i$. Define~$\mc{N}_i^{{\scriptsize \mbox{in}}}$ to be the collection of agent~$i$ itself and it's in-neighbors\footnote{Unlike related literature~\cite{opdirect_Nedic,opdirect_Tsianous,opdirect_Tsianous2,opdirect_Tsianous3,D-DGD,D-DPS,opdirect_Makhdoumi,DEXTRA,opdirect_nedicLinear,ADD-OPT}, we do not give the definition of out-neighbors,~$\mc{N}_i^{{\scriptsize \mbox{out}}}$, as such knowledge is not required.}, i.e., the set of agents that can send information to agent~$i$. We focus on solving an optimization problem that is distributed over the above multi-agent network. In particular, the agents cooperatively solve the following optimization problem:
\begin{align}
\mbox{P1}:
\quad\mbox{min  }&f(\mb{x})=\sum_{i=1}^nf_i(\mb{x}),\nonumber
\end{align}
where each local objective function,~$f_i:\mbb{R}^p\rightarrow\mbb{R}$, is convex and differentiable and known only to agent~$i$. Our goal is to develop a distributed iterative algorithm such that each agent converges to the global solution of Problem P1 given that the underlying graph,~$\mc{G}$, is directed.

Before we describe the algorithm, we formalize the set of assumptions we will use for our results. The following assumptions are standard in the literature concerning distributed and smooth optimization, see e.g.,~\cite{EXTRA,Augmented_EXTRA,DEXTRA,ADD-OPT,opdirect_nedicLinear}.
\begin{assump}\label{asp1}
	\begin{color}{black}The directed graph is strongly-connected and the agents have and know their unique identifiers, e.g.,~$1,\cdots,n$.\end{color}
\end{assump}
\begin{assump}\label{asp2}
	Each local function,~$f_i$, is differentiable and strongly-convex with Lipschitz continuous gradients, i.e., for any~$i$ and~$\mb{x}_1, \mb{x}_2\in\mbb{R}^p$,
	\begin{color}{black}
	\begin{enumerate}[label=(\alph*)]
		\item there exists a positive constant~$s$ such that
		$$f_i(\mb{x}_1)-f_i(\mb{x}_2)\leq\mb{\nabla} f_i(\mb{x}_1)^\top~(\mb{x}_1-\mb{x}_2)-\frac{s}{2}\|\mb{x}_1-\mb{x}_2\|_2^2;$$
		\item there exists a positive constant~$l$ such that
		$$\|\mb{\nabla} f_i(\mb{x}_1)-\mb{\nabla} f_i(\mb{x}_2)\|_2\leq l\|\mb{x}_1-\mb{x}_2\|_2.$$
	\end{enumerate}
    \end{color}
\end{assump}
\begin{color}{black}
	\noindent Clearly, the Lipschitz-continuity and strongly-convexity constants for the global objective function~$f(\mb{x})$ are~$nl$ and~$ns$, respectively. From the strong-convexity requirement, we know that the optimal solution exists and is unique and finite. 
\end{color}

To solve Problem P1, we propose the following algorithm. Each agent~$i$ maintains three vectors:~$\mb{x}_{k,i}$,~$\mb{z}_{k,i}\in\mbb{R}^p$, and~$\mb{y}_{k,i}\in\mbb{R}^n$, where~$k$ is the discrete-time index. At the~$k$th iteration, agent~$i$ performs the following updates:
\begin{subequations}\label{alg1}
	\begin{align}
	\mb{x}_{k+1,i}=&\sum_{j=1}^na_{ij}\mb{x}_{k,j}-\alpha\mb{z}_{k,i},\label{alg1a}\\
	\mb{y}_{k+1,i}=&\sum_{j=1}^na_{ij}\mb{y}_{k,j},\label{alg1b}\\
	\mb{z}_{k+1,i}=&\begin{color}{black}\sum_{j=1}^na_{ij}\mb{z}_{k,j}\end{color}+\frac{\nabla f_i(\mb{x}_{k+1,i})}{[\mb{y}_{k+1,i}]_i}-\frac{\nabla f_i(\mb{x}_{k,i})}{[\mb{y}_{k,i}]_i},\label{alg1d}
	\end{align}
\end{subequations}
where~$a_{ij}$ is the weight agent~$i$ assigns to agent~$j$,~$\alpha>0$ is a constant step-size,~$\nabla f_i(\mb{x}_{k,i})\in\mbb{R}^p$ is the gradient of~$f_i$ at~$\mb{x}_{k,i}$, and~$[\mb{y}_{k,i}]_i$ denotes the~$i$th element of~$\mb{y}_{k,i}$. Moreover, the weights are such that
\begin{align}
a_{ij}&=\left\{
\begin{array}{rl}
>0,&j\in\mc{N}_i^{{\scriptsize \mbox{in}}},\\
0,&\mbox{otw.},
\end{array}
\right.
\quad\qquad
\sum_{j=1}^na_{ij}=1,\forall i,\nonumber
\end{align}
i.e., the weight matrix~$A=\{a_{ij}\}$ is row-stochastic. For any agent~$i$, it is initiated with an arbitrary vector~$\mb{x}_{0,i}$ and with~$\mb{z}_{0,i}=\nabla f_i(\mb{x}_{0,i}),\mb{y}_{0,i}=\mb{e}_i$. 

In essence, the updates in Eqs.~\eqref{alg1a} and~\eqref{alg1d} form a modified version of the algorithm in~\cite{Augmented_EXTRA}\footnote{Ref.~\cite{Augmented_EXTRA} does not require this division as it considers doubly-stochastic matrices where the right eigenvector is~$\mb{1}_n^\top$ resulting in optimal consensus and/or optimization.}, where the gradients are scaled by the updates in Eq.~\eqref{alg1b}. \begin{color}{black}Note that the update in Eq.~\eqref{alg1b} converges asymptotically to the left eigenvector,~$\bs{\pi}^\top$, of the row-stochastic weight matrix; from Perron-Frobenius theorem~\cite{plemmons:79}. By scaling the gradients with the iterates from Eq.~\eqref{alg1b}, the imbalance, see Footnote~1, caused by the row-stochastic matrix is canceled.\end{color} The modification further enables the algorithm's convergence to the optimal solution with the collection of weights being only row-stochastic but not necessarily doubly-stochastic.

To simplify the notation, we assume, without loss of generality, that the sequences,~$\{\mb{x}_{k,i}\}$ and~$\{\mb{z}_{k,i}\}$, in Eq.~\eqref{alg1}, have only one dimension, i.e.,~$p=1$; thus~$x_{k,i}$,~$z_{k,i}$,~$\nabla f_i(x_{k,i})\in\mbb{R}$,~$\forall i,k$. The proof can be extended to arbitrary~$p$ with the Kronecker product notation. We next write Eqs.~\eqref{alg1} in a matrix form by defining~$\mb{x}_k$,~$\mb{z}_k$,~$\nabla\mb{f}_k\in\mbb{R}^n$ and~$Y_k\in\mbb{R}^{n\times n}$ as
\begin{align}
\mb{x}_k&=[x_{k,1},\cdots,x_{k,n}]^\top,\nonumber\\
\mb{z}_k&=[z_{k,1},\cdots,z_{k,n}]^\top,\nonumber\\
\nabla\mb{f}_k&=[\nabla f_1(x_{k,1}),\cdots,\nabla f_n(x_{k,n})]^\top,\nonumber\\
Y_k&=[\mb{y}_{k,1},\cdots,\mb{y}_{k,n}]^\top,\nonumber\\
\widetilde{Y}_k&=\mbox{diag}(Y_k).\nonumber
\end{align}
Recall that~$A$ is a row-stochastic collection of weights~$a_{ij}$. Given that the graph is strongly-connected and~$A$ is non-negative with positive diagonals, it follows that~$\widetilde{Y}_k$ is invertible for any~$k$. Based on the notation above, we write Eq.~\eqref{alg1} in the matrix form equivalently as follows:
\begin{subequations}\label{alg1_matrix}
	\begin{align}
	\mb{x}_{k+1}=&A\mb{x}_{k}-\alpha\mb{z}_{k},\label{alg1_ma}\\
	Y_{k+1}=&AY_{k},\label{alg1_mb}\\
	\mb{z}_{k+1}=&\begin{color}{black}A\mb{z}_{k}\end{color}+\widetilde{Y}_{k+1}^{-1}\nabla \mb{f}_{k+1}-\widetilde{Y}_k^{-1}\nabla \mb{f}_{k},\label{alg1_md}
	\end{align}
\end{subequations}
where~$\widetilde{Y}_0=I_n,\mb{z}_0=\nabla\mb{f}_0$, and~$\mb{x}_0$ is arbitrary.  From either Eqs. \eqref{alg1} or \eqref{alg1_matrix}, we emphasize that the implementation needs no knowledge of agent's out-degree for any agent in the network, which is in contrast to the existing work, e.g., in~\cite{opdirect_Nedic,opdirect_Tsianous,opdirect_Tsianous2,opdirect_Tsianous3,D-DGD,D-DPS,opdirect_Makhdoumi,DEXTRA,ADD-OPT,opdirect_nedicLinear}.

\section{Main Results}\label{s3}
In this section, we present the linear convergence result of Eq.~\eqref{alg1_matrix}. We further define
\begin{align}
Y_\infty&=\lim_{k\rightarrow\infty}Y_k=\lim_{k\rightarrow\infty}A^k,\nonumber\\
\mb{x}^*&=x^*\mb{1}_n,\nonumber\\
\widehat{\mb{x}}_k&=Y_\infty\mb{x}_k,\nonumber\\
\widehat{\mb{z}}_k&=Y_\infty\mb{z}_k,\nonumber\\
\nabla\mb{f}^*&=[\nabla f_1(x^*),\cdots,\nabla f_n(x^*)]^\top,\nonumber\\
\nabla\widehat{\mb{f}}_k&=\frac{1}{n}~\mb{1}_n\mb{1}_n^\top~\left[\nabla f_1(\widehat{x}_k),...,\nabla f_n(\widehat{x}_k)\right]^\top,\nonumber
\end{align}
where~$x^*$ denotes the optimal solution, see Assumption~\ref{asp2}. We denote constants~$\tau$,~$\epsilon$, and~$\eta$ as
\begin{align}
\tau&=\left\|A-I_n\right\|_2,\nonumber\\
\epsilon&=\left\|I_n-Y_\infty\right\|_2,\nonumber\\
\eta&=\max\left(\left|1-\alpha nl\right|,\left|1-\alpha ns\right|\right),\nonumber
\end{align}
where~$\alpha$ is the step-size, and~$l$ and~$s$ are respectively the Lipschitz gradient and strong-convexity constants in Assumption~\ref{asp2}. Let~$y$ and~$\widetilde{y}$ be defined as
\begin{align}
y&=\sup_{k}\left\|Y_k\right\|_2,\nonumber\\
\widetilde{y}&=\sup_{k}\left\|\widetilde{Y}_k^{-1}\right\|_2.\nonumber
\end{align}
\begin{color}{black}Since~$A$ is non-negative and irreducible, and~$Y_0$ is an~$n\times n$ identity matrix, we have that~$Y_k$ is convergent, and all of its diagonal elements are nonzero and bounded, for all~$k$. Therefore,~$\widetilde{y}$ is finite.\end{color} We now provide some lemmas that will be useful in the remainder of the paper. 
\begin{lem}\label{lem2}
	(Nedic \textit{et al}.~\cite{opdirect_Nedic}) Consider~$Y_k$, generated from the row-stochastic matrix,~$A$, and its limit~$Y_\infty$. There exist~$0<\gamma_1<1$ and~$0<T<\infty$ such that
	\begin{align}\label{DkDinfty1}
	\left\|Y_k-Y_\infty\right\|_2\leq T\gamma_1^{k},\qquad  \forall k.
	\end{align}
\end{lem}

\begin{color}{black}\begin{lem}\label{lem3}
	For any~$\mb{a}\in\mbb{R}^n$, define~$\widehat{\mb{a}}=Y_\infty\mb{a}$. Then there exists~$0<\sigma<1$ such that
	\begin{align}\label{sigma_eq}
	\left\|A\mb{a}-\widehat{\mb{a}}\right\|\leq\sigma\left\|\mb{a}-\widehat{\mb{a}}\right\|.
	\end{align}
\end{lem}
\begin{proof}
Since~$A$ is irreducible, row-stochastic with positive diagonals, from Perron-Frobenius theorem~\cite{plemmons:79} we note that~$\rho(A)=1$, every eigenvalue of~$A$ other than~$1$ is strictly less than~$\rho(A)$, and~$\bs{\pi}^\top$ is a strictly positive (left) eigenvector corresponding to the eigenvalue of~$1$ such that~$\bs{\pi}^\top\mb{1}_n = 1$; thus~$Y^\infty = \lim_{k\rightarrow\infty} A^k = \mb{1}_n\bs{\pi}^\top$. We now have
\begin{eqnarray}
AY_{\infty} &=& A\mb{1}_n\bs{\pi}^{\top} ~=~ \mb{1}_n\bs{\pi}^{\top} ~=~ Y_{\infty}, \nonumber \\
Y_{\infty}Y_{\infty} &=& \mb{1}_n\bs{\pi}^{\top}\mb{1}_n\bs{\pi}^{\top} ~=~ \mb{1}_n\bs{\pi}^{\top} ~=~ Y_{\infty}, \nonumber
\end{eqnarray}
and thus~$AY_{\infty}-Y_{\infty}Y_{\infty}$ is a zero matrix. Therefore,
\begin{eqnarray}
	A\mb{a}-Y_{\infty}\mb{a} &=& (A-Y_{\infty})(\mb{a}-Y_{\infty}\mb{a}) \nonumber.
\end{eqnarray}
Next we note that~$\rho(A-Y_{\infty})<1$ due to which there exists a matrix norm such that~$\| A-Y_{\infty}\| <1$ with a compatible vector norm,~$\|\cdot\|$, see~\cite{hornjohnson:13}: Chapter~5 for details, i.e.,
\begin{eqnarray}
\left\|A\mb{a}-Y_{\infty}\mb{a}\right\| &\leq& \|A-Y_{\infty}\| \left\|\mb{a}-Y_{\infty}\mb{a}\right\|,
\end{eqnarray}
and the lemma follows with~$\sigma=\|A-Y_{\infty}\|$.
\end{proof}\end{color}

	\begin{lem}\label{yy-}
		There exists some constant~$\widetilde{T}$ such that the following inequalities hold for all~$k\geq 1$.
		\begin{enumerate}[label=(\alph*)]
			\item
			$\left\|\widetilde{Y}_{k}^{-1}-\widetilde{Y}_{\infty}^{-1}\right\|_2\leq \widetilde{y}^2\widetilde{T}\gamma_1^{k}$
			\item 
			$\left\|\widetilde{Y}_{k+1}^{-1}-\widetilde{Y}_{k}^{-1}\right\|_2\leq 2\widetilde{y}^2\widetilde{T}\gamma_1^{k}$
		\end{enumerate}
	\end{lem}
	\begin{proof}
		The proof of (a) follows
		\begin{align}
		\left\|\widetilde{Y}_{k}^{-1}-\widetilde{Y}_{\infty}^{-1}\right\|_2&\leq\left\|\widetilde{Y}_{k}^{-1}\right\|_2\left\|\widetilde{Y}_{k}-\widetilde{Y}_{\infty}\right\|_2\left\|\widetilde{Y}_{\infty}^{-1}\right\|_2,\nonumber\\
		&\leq \widetilde{y}^2\widetilde{T}\gamma_1^{k},\nonumber
		\end{align}
		where the second inequality holds due to Lemma~\ref{lem2} and the fact that all matrix norms on finite-dimensional vector spaces are equivalent. The result in~(b) is straightforward by applying~(a), which completes the proof.
	\end{proof}

Based on the above discussion and notation, we finally denote~$\mb{t}_k$,~$\mb{s}_k\in\mbb{R}^3$, and~$G$,~$H_k\in\mbb{R}^{3\times3}$ for all~$k$ as
\begin{subequations} 
\begin{align}\label{t}
\mb{t}_k&=\left[
\begin{array}{c}
\left\|\mb{x}_k-\widehat{\mb{x}}_k\right\| \\
~\left\|\widehat{\mb{x}}_k-\mb{x}^*\right\|_2 \\
\left\|\mb{z}_k-\widehat{\mb{z}}_k\right\|
\end{array}
\right],
\qquad
\mb{s}_k=\left[
\begin{array}{cc}
\left\|\nabla\mb{f}_k\right\|_2 \\
0\\
0
\end{array}
\right],\\\label{G}
G&=\left[
\begin{array}{ccc}
\sigma & 0 &\alpha \\
\alpha cnl & \eta & 0\\
cd\epsilon\widetilde{y}l(\tau+\alpha nl)  & \alpha d\epsilon l^2\widetilde{y}n & \sigma+\alpha cd\epsilon l\widetilde{y}
\end{array}
\right],\\\label{Hk}
H_k&=\left[
\begin{array}{ccc}
0 & 0 & 0\\
(\alpha y\widetilde{y}^2\widetilde{T})\gamma_1^k & 0 & 0\\
(\alpha\epsilon\widetilde{y}l y+2\epsilon)d\widetilde{y}^2\widetilde{T}\gamma_1^k & 0 & 0
\end{array}
\right].
\end{align}
\end{subequations}
According to the definition of~$\mb{t}_k$ in Eq.~\eqref{t}, it is sufficient to show the linear convergence of~$\|\mb{t}_k\|_2$ to zero in order to prove that~$\mb{x}_k$ converges to~$\mb{x}^*$ linearly. The following theorem builds an inequality for analyzing~$\mb{t}_k$, which is an important result going forward.
\begin{theorem}\label{thm1}
	Let Assumptions~\ref{asp1} and~\ref{asp2} hold. We have 
	\begin{align}\label{thm1_eq}
	\mb{t}_{k+1}\leq G\mb{t}_{k}+H_{k}\mb{s}_{k},\qquad\forall k,
	\end{align}
\end{theorem}

The proof of theorem~\ref{thm1} is provided in Section~\ref{s4}. Note that Eq.~\eqref{thm1_eq} provides a linear inequality between~$\mb{t}_{k+1}$ and~$\mb{t}_{k}$ with matrices,~$G$ and~$H_k$. Thus, the convergence of~$\mb{t}_k$ is fully determined by~$G$ and~$H_k$. More specifically, in order to prove a linear convergence of~$\|\mb{t}_k\|_2$ going to zero, it is sufficient to show that~$\rho(G)<1$, where~$\rho(\cdot)$ denotes the spectral radius, as well as the linear decaying of~$H_k$, which is straightforward since~$0<\gamma_1<1$ (from Lemma~\ref{lem2}). In the following Lemma~\ref{lem_G}, we first show that with an appropriate step-size, the spectral radius of~$G$ is less than~$1$. We then show the linear convergence rate of~$G^k$ and~$H_k$ in Lemma~\ref{lem_GH}. We finally present the main result in Theorem~\ref{main_result}, which is based on these lemmas as well as Theorem~\ref{thm1}. 
\begin{lem}\label{lem_G}
	Consider the matrix~$G$ defined in Eq.~\eqref{G} as a function of the step-size,~$\alpha$, \begin{color}{black}denoted in this lemma as~$G_\alpha$ to motivate this dependence.\end{color} It follows that~$\rho(G_\alpha)<1$ if the step-size,~$\alpha\in(0,\overline{\alpha})$, where
	\begin{color}{black}
	\begin{align}\label{alpha_ub}
	\overline{\alpha}=\min\left\{\frac{\sqrt{\Delta^2+4cd\epsilon\widetilde{y}n^3l^2(l+s)s(1-\sigma)^2}-\Delta}{2cd\epsilon\widetilde{y}n^2l^2(l+s)},\frac{1}{nl}\right\},
	\end{align}
	and~$\Delta=cd\epsilon\widetilde{y}lns(\tau+1-\sigma)$,
	where~$c$ and~$d$ are the constants from the equivalence of~$\|\cdot\|$ defined in Lemma~\ref{lem3} and~$\|\cdot\|_2$. 
    \end{color}
\end{lem}
\begin{proof}
First, note that if~$\alpha<\frac{1}{nl}$ then~$\eta=1-\alpha ns,$ since~$l\geq s$ (see e.g.,~\cite{opt_literature0}: Chapter 3 for details). Setting~$\alpha=0$, we get
	\begin{align}
	G_0&=\left[
	\begin{array}{ccc}
	\sigma & 0 & 0 \\
	0 & 1 & 0\\
	cd\epsilon l\tau \widetilde{y} & 0 & \sigma
	\end{array}
	\right],
	\end{align}
	the eigenvalues of which are~$\sigma$,~$\sigma$, and~$1$. Note that~$0<\sigma<1$. Therefore,~$\rho(G_0)=1$, where~$\rho(\cdot)$ denotes the spectral radius. We now consider how does the spectral radius,~$\rho(G_\alpha)$, change if we slightly increase~$\alpha$ from 0. To this aim, we consider the characteristic polynomial of~$G_\alpha$ and let it go to zero, i.e., with~$\mbox{det}(qI_n-G_\alpha)=0,q\in\mbb{C}$, we have
	\begin{align}\label{det}
		\left((q-\sigma)^2-\alpha cd\epsilon l\widetilde{y}(q-\sigma)\right)(q-1+n\alpha s)\nonumber\\
		-\alpha cd(q-1+n\alpha s)\left(\epsilon l\tau \widetilde{y}+\alpha(\epsilon l^2\widetilde{y}n)\right)-\alpha^3cdn^2l^3\epsilon \widetilde{y}=0.
	\end{align} 
	Eq.~\eqref{det} shows an implicit relation between~$q$ and~$\alpha$. Since~$1$ is an eigenvalue of~$G_0$, Eq.~\eqref{det} holds when~$q=1$ and~$\alpha=0$. By taking the derivative of both sides of Eq.~\eqref{det} with~$q=1,\alpha=0$, we obtain that \textcolor{black}{$\frac{d q}{d\alpha}|_{\alpha=0,q=1}=-ns<0$}. In other words, when~$\alpha$ increases slightly from~$0$,~$\lambda_{1}(G_0)=1$ decreases\footnote{Note that a decrease in~$\alpha$ from~$0$ makes the eigenvalue of~$1$ increase, rendering a negative step-size infeasible for convergence, justifying the choice of a strictly positive step-size.}. Since~$\lambda_{2,3}(G_0)=\sigma<1$, we obtain that~$\rho(G_\alpha)<1$ when~$\alpha$ is slightly increased from~$0$ because of the continuity of eigenvalues as a function of the matrix elements.
	
\begin{color}{black}The next step is to find the range of step-sizes,~$(0,\overline{\alpha})$, such that all three eigenvalues are less than~$1$. To this aim, let~$q=1$ in Eq.~\eqref{det}, then by solving for~$\alpha$, we obtain the values of~$\alpha$ for which there will be an eigenvalue of~$1$. These values are the three roots of Eq.~\eqref{det} and are~$\alpha_1=0$, some~$\alpha_2<0$ (not applicable as with negative $\alpha$, there exists an eigenvalue that is greater than~$1$, see Footnote 4), and
\begin{color}{black}\[
\alpha_3=\frac{\sqrt{\Delta^2+4cd\epsilon\widetilde{y}n^3l^2(l+s)s(1-\sigma)^2}-\Delta}{2cd\epsilon\widetilde{y}n^2l^2(l+s)}>0,
\]  
where~$\Delta=cd\epsilon\widetilde{y}lns(\tau+1-\sigma)$. \end{color} Now note that when~$\alpha$ is zero, we have an eigenvalue of~$1$ and two eigenvalues of~$\sigma<1$. At~$\alpha_3>0$, we have an eigenvalue of~$1$ and there is no~$\alpha\in(0,\overline{\alpha})$ that leads to~$\lambda(G_\alpha)=1$ because no solution of Eq.~\eqref{det} lies in this interval. Hence, we conclude that all eigenvalues of~$G_\alpha$ are less than~$1$ for~$\alpha\in(0,\overline{\alpha})$, where~$\overline{\alpha}$ is defined in the Lemma's statement, by considering the fact that eigenvalues are continuous functions of a matrix. Therefore,~$\rho(G_\alpha)<1$, when~$\alpha\in(0,\overline{\alpha})$.\end{color}
\end{proof}
\begin{lem}\label{lem_GH}
	With the step-size,~$\alpha\in(0,\overline{\alpha})$, where~$\overline{\alpha}$ is defined in Eq.~\eqref{alpha_ub}, the following statements hold for all~$k$,
	\begin{enumerate}[label=(\alph*)]
		\item
		there exist~$0<\gamma_1<1$ and~$0<\Gamma_1<\infty$, where~$\gamma_1$ is defined in Eq.~\eqref{DkDinfty1}, such that
		$$\left\|H_k\right\|_2=\Gamma_1\gamma_1^k;$$
		\item 
		there exist~$0<\gamma_2<1$ and~$0<\Gamma_2<\infty$, such that
		$$\left\|G^k\right\|_2\leq\Gamma_2\gamma_2^k;$$
		\item 
		Let~$\gamma=\max\{\gamma_1,\gamma_2\}$ and~$\Gamma=\Gamma_1\Gamma_2/\gamma$. Then for all~$0\leq r\leq k-1$,
		$$\left\|G^{k-r-1}H_r\right\|_2\leq\Gamma\gamma^k.$$
	\end{enumerate}
\end{lem}
\begin{proof}
	(a) This is easy to verify according to Eq.~\eqref{t}, and by letting
	\begin{color}{black}
	$\Gamma_1=\sqrt{(\alpha y\widetilde{y}^2\widetilde{T})^2+(\epsilon d\widetilde{y}^2\widetilde{T})^2(2+\alpha y\widetilde{y}l)^2}$.
	\end{color}
	
	(b) Since the spectral radius of~$G$ is less than one, there exists some matrix norm of~$G$ that is also less than one. We let the value of this matrix norm of~$G$ to be~$\gamma_2$. Then, from the equivalence of norms, we have
	{\color{black}
	\begin{align}
	\left\|G^k\right\|_2\leq\Gamma_2\gamma_2^k,
	\end{align}
	for some positive~$\Gamma_2$.
    }
    
	(c) The proof of (c) follows from combining (a) and (b). 
\end{proof}
\begin{lem}\label{lem_polyak}
	(Polyak~\cite{polyak1987introduction}) If nonnegative sequences~$\{v_k\}$,~$\{u_k\}$,~$\{b_k\}$ and~$\{c_k\}$ are such that ~$\sum_{k=0}^{\infty} b_k < \infty$,~$\sum_{k=0}^{\infty} c_k < \infty$ and
	$$v_{k+1} \le (1+b_k)v_k - u_k + c_k, \quad \forall t\ge 0,$$ then~$\{v_k\}$ converges and~$\sum_{k=0}^{\infty} u_k < \infty$.
\end{lem}
To recap, we provide the linear iterative relation on~$\mb{t_k}$ with matrices~$G$ and~$H_k$ in Theorem 1. Subsequently, we show that~$\rho(G)<1$ and the linear decaying of~$H_k$ in Lemmas~\ref{lem_G} and~\ref{lem_GH} for sufficiently small step-size,~$\alpha$. By combining these relations and Lemma~\ref{lem_polyak}, we are ready to prove the linear convergence of our algorithm in the following theorem. 

\begin{theorem}\label{main_result}
	Let Assumptions~\ref{asp1} and~\ref{asp2} hold. With the step-size,~$\alpha\in(0,\overline{\alpha})$, where~$\overline{\alpha}$ is defined in Eq.~\eqref{alpha_ub}, the sequence,~$\{\mb{x}_k\}$, generated by Eq.~\eqref{alg1_matrix}, converges linearly to the optimal solution,~$\mb{x}^*$, i.e., there exist some constant~$M>0$ such that
	\begin{align}
	\left\|\mb{x}_k-\mb{x}^*\right\|_2\leq M(\gamma+\xi)^k,\qquad\forall k,
	\end{align}
	where~$\xi$ is a arbitrarily small constant.
\end{theorem}
\begin{proof}
	We write Eq.~\eqref{thm1_eq} recursively, which results
	\begin{align}\label{thm2_eq1}
	\mb{t}_k\leq&G^k\mb{t}_0+\sum_{r=0}^{k-1}G^{k-r-1}H_r\mb{s}_r.
	\end{align}
	By taking two-norm on both sides of Eq.~\eqref{thm2_eq1}, and considering Lemma~\ref{lem_GH}, we obtain that
	\begin{align}\label{thm2_eq2}
	\left\|\mb{t}_k\right\|_2\leq&\left\|G^k\right\|_2\left\|\mb{t}_0\right\|_2+\sum_{r=0}^{k-1}\left\|G^{k-r-1}H_r\right\|_2\left\|\mb{s}_r\right\|_2,\nonumber\\
	\leq&\Gamma_2\gamma_2^k\left\|\mb{t}_0\right\|_2+\sum_{r=0}^{k-1}\Gamma\gamma^k\left\|\mb{s}_r\right\|_2,
	\end{align}
	in which we can bound~$\|\mb{s}_r\|_2$ as
	\begin{align}
	\left\|\mb{s}_r\right\|_2\leq&\left\|\nabla\mb{f}_r-\nabla\mb{f}^*\right\|_2+\left\|\nabla\mb{f}^*\right\|_2,\nonumber\\
	\leq&l\left\|\mb{x}_r-\widehat{\mb{x}}_r\right\|_2+l\left\|\widehat{\mb{x}}_r-\mb{x}^*\right\|_2+\left\|\nabla\mb{f}^*\right\|_2,\nonumber\\
	\leq&(c+1)l\left\|\mb{t}_r\right\|_2+\left\|\nabla\mb{f}^*\right\|_2.
	\end{align}
	Therefore, we have that for all~$k$
	\begin{align}\label{thm2_eq3}
	\left\|\mb{t}_k\right\|_2\leq&\bigg{(}\Gamma_2\|\mb{t}_0\|_2+(c+1)\Gamma  l\sum_{r=0}^{k-1}\|\mb{t}_r\|_2+\Gamma k\|\nabla\mb{f}^*\|_2\bigg{)}\gamma^k.
	\end{align}
	Denote~$v_k=\sum_{r=0}^{k-1}\|\mathbf{t}_r\|_2$,~$s_k=\Gamma_2\|\mathbf{t}_0\|_2+\Gamma k\|\nabla \mathbf{f}^*\|_2$, and~$b=(c+1)\Gamma l$, then Eq. \eqref{thm2_eq3} results into
	\begin{align}\label{thm2_newEq}
	\|\mathbf{t}_{k}\|_2 = v_{k+1} - v_k\leq(s_k + bv_k)\gamma^k,
	\end{align}
	which implies that~$v_{k+1} \le (1+b\gamma^k)v_k + s_k\gamma^k$. Applying Lemma~\ref{lem_polyak} with~$b_k = b\gamma^k$ and~$c_k = s_k\gamma^k$ (here~$u_k=0$), we have that~$v_k$ converges\footnote{In order to apply Lemma~\ref{lem_polyak}, we need to show that~$\sum_{k=0}^\infty s_k\gamma^k<\infty$, which follows from the fact that~$\lim_{k\rightarrow\infty}\frac{s_{k+1}\gamma^{k+1}}{s_{k}\gamma^{k}}=\gamma<1.$}.
    Moreover, since~$\{v_k\}$ is bounded, by Eq.~\eqref{thm2_newEq},~$\forall \mu \in (\gamma,1)$ we have
    \begin{equation}
    	\lim_{k\rightarrow\infty}\frac{\|\mathbf{t}_{k}\|_2}{\mu^k} \leq
    	\lim_{k\rightarrow\infty}\frac{(s_k + bv_k)\gamma^k}{\mu^k}=0.
    \end{equation} 
    Therefore,~$\|\mathbf{t}_{k}\|_2=O(\mu^k)$. In other words, there exists some positive constant~$\Phi$ such that for all~$k$, we have: 
	\begin{align}\label{thm2_eq6}
	\left\|\mb{t}_k\right\|_2\leq&\Phi(\gamma+\xi)^{k},
	\end{align}
	where~$\xi$ is a arbitrarily small constant.
	It follows that
	$\left\|\mb{x}_k-\mb{x}^*\right\|_2\leq\left\|\mb{x}_k-\widehat{\mb{x}}_k\right\|_2+\left\|\widehat{\mb{x}}_k-\mb{x}^*\right\|_2\leq (c+1)\left\|\mb{t}_k\right\|_2 \leq(c+1)\Phi(\gamma+\xi)^{k}$, which completes the proof.
\end{proof}
\begin{color}{black}\noindent Theorem~\ref{main_result} shows the linear convergence rate of our algorithm, the result of which is based on the iterative relation in Theorem~\ref{thm1}. The proposed algorithm works for a small enough step-size. However, the upper bound,~$\overline{\alpha}$, of this step-size is a function of the network parameters and cannot be computed locally. This notion of sufficiently small step-sizes is not uncommon in the literature where the step-size upper bound can be estimated in most circumstances, see e.g.,~\cite{uc_Nedic, EXTRA, opdirect_nedicLinear, ADD-OPT}. Furthermore, each agent must agree on the same value of step-size that may be pre-programmed to avoid implementing an agreement protocol. In next section, we present the proof of Theorem~\ref{thm1}.\end{color}

\section{Proof of Theorem~\ref{thm1}}\label{s4}
We first provide a few relevant auxiliary relations.

\subsection{Auxiliary Relations}
In the following, Lemma~\ref{w-x-} derives iterative equations that govern the sequences~$\widehat{\mb{z}}_k$ and~$\widehat{\mb{x}}_{k}$. Next, Lemma~\ref{gd_approach} is a standard result in the optimization literature,~\cite{opt_literature0}. It states that if we perform a gradient-descent step with a fixed step-size for a strongly-convex and smooth function, then the distance to optimizer shrinks by at least a fixed ratio. 
\begin{lem}\label{w-x-}
	The following equations hold for all~$k$:
	\begin{enumerate}[label=(\alph*)]
		\item
		$\widehat{\mb{z}}_k=Y_\infty\widetilde{Y}_k^{-1}\nabla\mb{f}_k$;
		\item 
		$\widehat{\mb{x}}_{k+1}=\widehat{\mb{x}}_k-\alpha\widehat{\mb{z}}_k$.
	\end{enumerate}
\end{lem}
\begin{proof}
	By noting that~$Y_\infty A=Y_\infty$, we obtain from Eq. \eqref{alg1_md}
	\begin{align}
	\widehat{\mb{z}}_k&=\widehat{\mb{z}}_{k-1}+Y_\infty\widetilde{Y}_k^{-1}\nabla\mb{f}_k-Y_\infty\widetilde{Y}_{k-1}^{-1}\nabla\mb{f}_{k-1}.\nonumber
	\end{align}
	Do this recursively, and we have that
	\begin{align}
	\widehat{\mb{z}}_k=\widehat{\mb{z}}_0+Y_\infty\widetilde{Y}_k^{-1}\nabla\mb{f}_k-Y_\infty\widetilde{Y}_0^{-1}\nabla\mb{f}_{0}.\nonumber
	\end{align}
	Recall the initial condition~$\mb{z}_0=\nabla\mb{f}_0=\widetilde{Y}_0^{-1}\nabla\mb{f}_{0}$, since~$\widetilde{Y}_0=I_n$. Thus,~$\widehat{\mb{z}}_0=Y_\infty\widetilde{Y}_0^{-1}\nabla\mb{f}_{0}$. Then we obtain the result of (a). The proof of (b) follows directly from Eq.~\eqref{alg1_ma}; in particular,
	\begin{align}
	\widehat{\mb{x}}_{k+1}&=Y_\infty\left(A\mb{x}_{k}-\alpha\mb{z}_k\right)\nonumber\\
	&=\widehat{\mb{x}}_{k}-\alpha\widehat{\mb{z}}_{k},\nonumber
	\end{align}
	which completes the proof.
\end{proof}
\begin{lem}\label{gd_approach}
\begin{color}{black}	(Bubeck~\cite{opt_literature0}) Let Assumption~\ref{asp1} hold for the objective function,~$f(x)$, in P1, where~$ns$ and~$nl$ are the strong-convexity constant and Lipschitz continuous gradient constant, respectively. For any~$x\in\mbb{R}$, define~$x_+=x-\alpha\nabla f(x)$, such that $0<\alpha<\frac{2}{nl}$. Then~
$$\left\|x_+-x^*\right\|_2~\leq~\eta\left\|x-x^*\right\|_2,$$ where~$\eta=\max\left(\left|1-\alpha nl\right|,\left|1-\alpha ns\right|\right)$.\end{color}
\end{lem}

\subsection{Proof of Theorem~\ref{thm1}}
We now provide the proof of Theorem~\ref{thm1}. To this aim, we will bound~$\|\mb{x}_{k+1}-\widehat{\mb{x}}_{k+1}\|$,~$\|\widehat{\mb{x}}_{k+1}-\mb{x}^*\|_2$, and~$\|\mb{z}_{k+1}-\widehat{\mb{z}}_{k+1}\|$ by the linear combinations of their past values, i.e.,~$\|\mb{x}_{k}-\widehat{\mb{x}}_{k}\|$,~$\|\widehat{\mb{x}}_{k}-\mb{x}^*\|_2$, and~$\|\mb{z}_{k}-\widehat{\mb{z}}_{k}\|$, as well as~$\|\nabla\mb{f}_k\|$. The coefficients will be shown to be the entries of~$G$ and $H_{k}$.

\textbf{Step 1:} Bounding~$\|\mb{x}_{k+1}-\widehat{\mb{x}}_{k+1}\|$. \\
According to Eq.~\eqref{alg1_ma} and Lemma~\ref{w-x-}(b), we obtain that
\begin{align}
\left\|\mb{x}_{k+1}-\widehat{\mb{x}}_{k+1}\right\|\leq&\left\|A\mb{x}_{k}-\widehat{\mb{x}}_{k}\right\|+\alpha\left\|\mb{z}_{k}-\widehat{\mb{z}}_{k}\right\|.
\end{align}
By noting that~$\|A\mb{x}_{k}-\widehat{\mb{x}}_{k}\|\leq\sigma\|\mb{x}_{k}-\widehat{\mb{x}}_{k}\|$ from Lemma~\ref{lem3}, we have
\begin{align}\label{step1}
\left\|\mb{x}_{k+1}-\widehat{\mb{x}}_{k+1}\right\|\leq&\sigma\left\|\mb{x}_{k}-\widehat{\mb{x}}_{k}\right\|+\alpha\left\|\mb{z}_{k}-\widehat{\mb{z}}_{k}\right\|.
\end{align}

\textbf{Step 2:} Bounding~$\|\widehat{\mb{x}}_{k+1}-\mb{x}^*\|_2$. \\
By considering Lemma~\ref{w-x-}(b), we obtain that
\begin{align}\label{step2_eq1}
\left\|\widehat{\mb{x}}_{k+1}-\mb{x}^*\right\|_2\leq&\left\|\widehat{\mb{x}}_{k}-\alpha n\nabla\widehat{\mb{f}}_k-\mb{x}^*\right\|_2+\alpha\left\|\widehat{\mb{z}}_{k}-n\nabla\widehat{\mb{f}}_k\right\|_2.
\end{align}
Let~$\mb{x}_+=\widehat{\mb{x}}_{k}-\alpha(n\nabla\widehat{\mb{f}}_k)$. Considering the definition of~$\nabla\widehat{\mb{f}}_k$, the updates of~$\mb{x}_+$ performs a (centralized) gradient descent with step-size~$\alpha$ to minimize the objective function in Problem P1. Therefore, we have that, according to Lemma~\ref{gd_approach},
\begin{align}
\left\|\widehat{\mb{x}}_{k}-\alpha n\nabla\widehat{\mb{f}}_k-\mb{x}^*\right\|_2=\left\|\mb{x}_+-\mb{x}^*\right\|_2\leq\eta\left\|\widehat{\mb{x}}_{k}-\mb{x}^*\right\|_2.
\end{align}
We next bound the second term in the RHS of Eq.~\eqref{step2_eq1} by splitting it such that
\begin{align}\label{step2_eq2}
\left\|\widehat{\mb{z}}_{k}-n\nabla\widehat{\mb{f}}_k\right\|_2\leq&\left\|\widehat{\mb{z}}_{k}-Y_\infty\widetilde{Y}_\infty^{-1}\nabla\mb{f}_k\right\|_2\nonumber\\
&+\left\|Y_\infty\widetilde{Y}_\infty^{-1}\nabla\mb{f}_k-n\nabla\widehat{\mb{f}}_k\right\|_2.
\end{align}
The first term on the RHS of Eq.~\eqref{step2_eq2} is bounded by
\begin{align}\label{step2_eq3}
\left\|\widehat{\mb{z}}_{k}-Y_\infty\widetilde{Y}_\infty^{-1}\nabla\mb{f}_k\right\|_2&=\left\|Y_\infty\widetilde{Y}_k^{-1}\nabla\mb{f}_k-Y_\infty\widetilde{Y}_\infty^{-1}\nabla\mb{f}_k\right\|_2\nonumber\\
&\leq y\widetilde{y}^2\widetilde{T}\gamma_1^{k}\left\|\nabla\mb{f}_k\right\|_2,
\end{align}
where in the first equality we apply the result of Lemma~\ref{w-x-}(a) that~$\widehat{\mb{z}}_k=Y_\infty\widetilde{Y}_k^{-1}\nabla\mb{f}_k$, and in the second inequality we use the result of Lemma~\ref{yy-}(a). The second term on the RHS of Eq.~\eqref{step2_eq2} is bounded by
\begin{align}\label{step2_eq4}
\left\|Y_\infty\widetilde{Y}_\infty^{-1}\nabla\mb{f}_k-n\nabla\widehat{\mb{f}}_k\right\|_2&\leq \left\|\mb{1}_n\mb{1}_n^\top\right\|_2l\left\|\mb{x}_k-\widehat{\mb{x}}_k\right\|_2,\nonumber\\
&=nl\left\|\mb{x}_k-\widehat{\mb{x}}_k\right\|_2
\end{align}
where in the first inequality we use the relation that~$Y_\infty\widetilde{Y}_\infty^{-1}=\mb{1}_n\mb{1}_n^\top$. By combining Eqs.~\eqref{step2_eq2}, \eqref{step2_eq3}, and \eqref{step2_eq4}, it follows that
\begin{align}\label{step2_eq5}
\left\|\widehat{\mb{z}}_{k}-n\nabla\widehat{\mb{f}}_k\right\|_2\leq nl\left\|\mb{x}_k-\widehat{\mb{x}}_k\right\|_2+y\widetilde{y}^2\widetilde{T}\gamma_1^{k}\left\|\nabla\mb{f}_k\right\|_2.
\end{align}
Therefore, we can bound~$\|\widehat{\mb{x}}_{k+1}-\mb{x}^*\|_2$ as
\begin{align}\label{step2}
\left\|\widehat{\mb{x}}_{k+1}-\mb{x}^*\right\|_2\leq&\alpha cnl\left\|\mb{x}_k-\widehat{\mb{x}}_k\right\|+\eta\left\|\widehat{\mb{x}}_{k}-\mb{x}^*\right\|_2\nonumber\\
&+\alpha y\widetilde{y}^2\widetilde{T}\gamma_1^{k}\left\|\nabla\mb{f}_k\right\|_2.
\end{align}
\textbf{Step 3:} Bounding~$\|\mb{z}_{k+1}-\widehat{\mb{z}}_{k+1}\|$. \\
According to Eq.~\eqref{alg1_md}, we have
\begin{align}
&\left\|\mb{z}_{k+1}-\widehat{\mb{z}}_{k+1}\right\|\leq\left\|A\mb{z}_{k}-\widehat{\mb{z}}_{k}\right\|\nonumber\\
&+\left\|\left(\widetilde{Y}_{k+1}^{-1}\nabla \mb{f}_{k+1}-\widetilde{Y}_{k}^{-1}\nabla \mb{f}_{k}\right)-\left(\widehat{\mb{z}}_{k+1}-\widehat{\mb{z}}_{k}\right)\right\|.\label{step3_eq2}
\end{align}
With the result of Lemma~\ref{lem3}, we obtain that
\begin{align}\label{step3_eq22}
\left\|A\mb{z}_{k}-\widehat{\mb{z}}_{k}\right\|&\leq\sigma\left\|\mb{z}_{k}-\widehat{\mb{z}}_{k}\right\|.
\end{align}
Note that~$\widehat{\mb{z}}_k=Y_\infty\widetilde{Y}_k^{-1}\nabla\mb{f}_k$ from Lemma~\ref{w-x-}(a). Therefore,
\begin{align}
&\left\|\left(\widetilde{Y}_{k+1}^{-1}\nabla \mb{f}_{k+1}-\widetilde{Y}_{k}^{-1}\nabla \mb{f}_{k}\right)-\left(\widehat{\mb{z}}_{k+1}-\widehat{\mb{z}}_{k}\right)\right\|_2\nonumber\\
&=\left\|\left(I_n-Y_\infty\right)\left(\widetilde{Y}_{k+1}^{-1}\nabla \mb{f}_{k+1}-\widetilde{Y}_{k}^{-1}\nabla \mb{f}_{k}\right)\right\|_2,\nonumber\\
&\leq\epsilon\left\|\widetilde{Y}_{k+1}^{-1}\nabla \mb{f}_{k+1}-\widetilde{Y}_{k+1}^{-1}\nabla\mb{f}_{k}\right\|_2+\epsilon\left\|\widetilde{Y}_{k+1}^{-1}\nabla\mb{f}_{k}-\widetilde{Y}_{k}^{-1}\nabla\mb{f}_{k}\right\|_2,\nonumber\\
&\leq\epsilon\widetilde{y}l\left\|\mb{x}_{k+1}-\mb{x}_{k}\right\|_2+2\epsilon \widetilde{y}^2\widetilde{T}\gamma_1^{k}\left\|\nabla\mb{f}_{k}\right\|_2.\label{step3_eq3}
\end{align}
We now bound~$\|\mb{x}_{k+1}-\mb{x}_{k}\|_2$ in Eq.~\eqref{step3_eq3}. Note that~$(A-I_n)\widehat{\mb{x}}_k=\mb{0}_n$ for all~$k$, which results into
\begin{align}\label{step3_eq5}
\left\|\mb{x}_{k+1}-\mb{x}_{k}\right\|_2\leq&\left\|(A-I_n)\mb{x}_k\right\|_2+\alpha\left\|\mb{z}_k\right\|_2,\nonumber\\
\leq&\left\|(A-I_n)\left(\mb{x}_k-\widehat{\mb{x}}_k\right)\right\|_2+\alpha\left\|\mb{z}_k\right\|_2,\nonumber\\
\leq&\tau\left\|\mb{x}_k-\widehat{\mb{x}}_k\right\|_2+\alpha\left\|\mb{z}_k\right\|_2,
\end{align}
where~$\|\mb{z}_k\|_2$ can be bounded with the following derivation:
\begin{align}\label{step3_eq6}
\left\|\mb{z}_k\right\|_2\leq&\left\|\mb{z}_k-\widehat{\mb{z}}_k\right\|_2+\left\|\widehat{\mb{z}}_k-n\nabla\widehat{\mb{f}}_k\right\|_2+\left\|n\nabla\widehat{\mb{f}}_k-\mb{1}_n\mb{1}_n^\top\nabla\mb{f}^*\right\|_2,\nonumber\\
\leq&\left\|\mb{z}_k-\widehat{\mb{z}}_k\right\|_2+nl\left\|\mb{x}_k-\widehat{\mb{x}}_k\right\|_2+y\widetilde{y}^2\widetilde{T}\gamma_1^{k}\left\|\nabla\mb{f}_k\right\|_2\nonumber\\
&+nl\left\|\widehat{\mb{x}}_k-\mb{x}^*\right\|_2,
\end{align}
where we use the fact that~$\mb{1}_n\mb{1}_n^\top\nabla\mb{f}^*=\mb{0}_n$ and we bound~$\|\widehat{\mb{z}}_k-n\nabla\widehat{\mb{f}}_k\|_2$ using the result in Eq. \eqref{step2_eq5}.
By combining Eqs. \eqref{step3_eq2}, \eqref{step3_eq22}, \eqref{step3_eq3}, \eqref{step3_eq5}, and \eqref{step3_eq6}, we finally get that
\begin{align}\label{step3}
\left\|\mb{z}_{k+1}-\widehat{\mb{z}}_{k+1}\right\|\leq&\left(\alpha cd\epsilon\widetilde{y}nl^2+cd\epsilon\widetilde{y}l\tau\right)\left\|\mb{x}_k-\widehat{\mb{x}}_k\right\|\nonumber\\
&+\alpha d\epsilon\widetilde{y}nl^2\left\|\widehat{\mb{x}}_k-\mb{x}^*\right\|_2\nonumber\\
&+\left(\sigma+\alpha cd\epsilon\widetilde{y}l\right)\left\|\mb{z}_k-\widehat{\mb{z}}_k\right\|\nonumber\\
&+(\alpha\epsilon\widetilde{y}ly+2\epsilon)d\widetilde{y}^2\widetilde{T}\gamma_1^k\left\|\nabla\mb{f}_k\right\|_2.
\end{align}

\textbf{Step 4:} By combining Eqs.~\eqref{step1} in step 1, \eqref{step2} in step 2, and \eqref{step3} in step 3, we complete the proof.

\section{Numerical Experiments}\label{s5}
In this section, we verify the performance of the proposed algorithm. We compare the convergence rates between our algorithm and other existing methods, including DEXTRA~\cite{DEXTRA}, ADD-OPT ~\cite{ADD-OPT}, Push-DIGing~\cite{opdirect_nedicLinear}, Subgradient-Push~\cite{opdirect_Nedic}, Directed-Distributed Subgradient Descent~\cite{D-DGD}, and the Weight-Balancing Distributed Subgradient Descent~\cite{opdirect_Makhdoumi}. Our numerical experiments are based on the distributed logistic regression problem over a directed graph:
\begin{align}
\mb{z}^*=\underset{\mb{z}\in\mbb{R}^p}{\operatorname{argmin}}\frac{\beta}{2}\|\mb{z}\|^2+\sum_{i=1}^n\sum_{j=1}^{m_i}\ln\left[1+\exp\left(-\left(\mb{c}_{ij}^\top\mb{z}\right)b_{ij}\right)\right],\nonumber
\end{align}
where for any agent~$i$, it has access to~$m_i$ training examples,~$(\mb{c}_{ij},b_{ij})\in\mbb{R}^p\times\{-1,+1\}$, where~$\mb{c}_{ij}$ includes the~$p$ features of the~$j$th training example of agent~$i$, and~$b_{ij}$ is the corresponding label. The directed graph is shown in Fig.~\ref{fig1} This problem can be formulated in the form of P1 with the private objective function~$f_i$ being
\begin{align}
f_i=\frac{\beta}{2n}\|\mb{z}\|^2+\sum_{j=1}^{m_i}\ln\left[1+\exp\left(-\left(\mb{c}_{ij}^\top\mb{z}\right)b_{ij}\right)\right].\nonumber
\end{align}
\begin{figure}[!h]
	\begin{center}
		\noindent
		\includegraphics[width=1.8in]{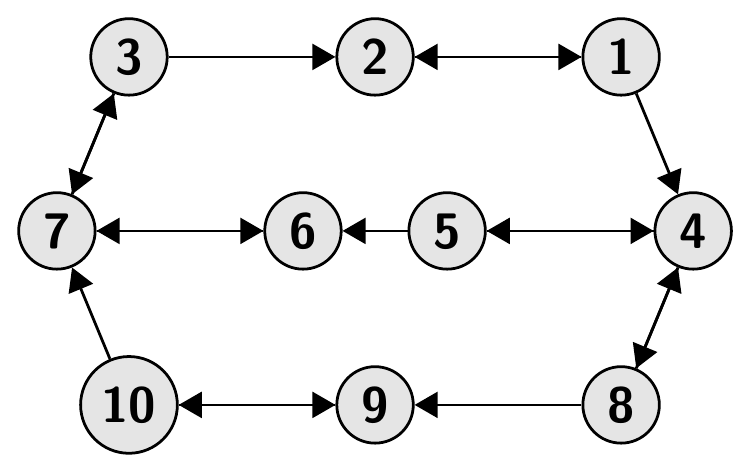}
		\caption{A strongly-connected directed network.}\label{fig1}
	\end{center}
\end{figure}
\begin{color}{black}
In our setting, we have~$n=10$,~$m_i=10$, for all~$i$, and~$p=3$. In the implementation, we apply the same local degree weighting strategy to all methods. Fig.~\ref{fig2} shows the convergence rates of all methods with appropriate step-sizes. The step-sizes used for Subgradient-Push, Directed-Distributed Subgradient Descent, and the Weight-Balancing Distributed Subgradient Descent are on the order of~$\alpha_k=O(\frac{1}{\sqrt{k}})$ at the~$k$th iteration. For the proposed algorithm in this paper, DEXTRA, and ADD-OPT, the corresponding constant step-sizes are~$\alpha=0.008$,~$\alpha=0.1$, and~$\alpha=0.03$, respectively. Note that this paper and ADD-OPT converge for a small enough positive step-size while DEXTRA requires a (strictly) positive lower bound. 

We note here that Subgradient-Push~\cite{opdirect_Nedic}, Directed-Distributed Subgradient Descent~\cite{D-DGD}, and the Weight-Balancing Distributed Subgradient Descent~\cite{opdirect_Makhdoumi} converge for general convex functions, while the algorithm in this paper, DEXTRA, ADD-OPT, and Push-DIGing work only for strongly-convex functions. Despite this fact, the algorithms in~\cite{opdirect_Nedic,D-DGD,opdirect_Makhdoumi} converge at a sublinear rate. In contrast, the algorithm in this paper, as well as ADD-OPT, Push-DIGing, and DEXTRA, has a fast linear convergence rate. Compared to ADD-OPT, Push-DIGing, and DEXTRA, our algorithm requires no knowledge of agents out-degree, which is more practical in certain communication networks, albeit we require that the agents know and have unique identifiers.
\end{color} 
\begin{figure}[!h]
	\begin{center}
		\noindent
		\subfigure{\includegraphics[width=3in]{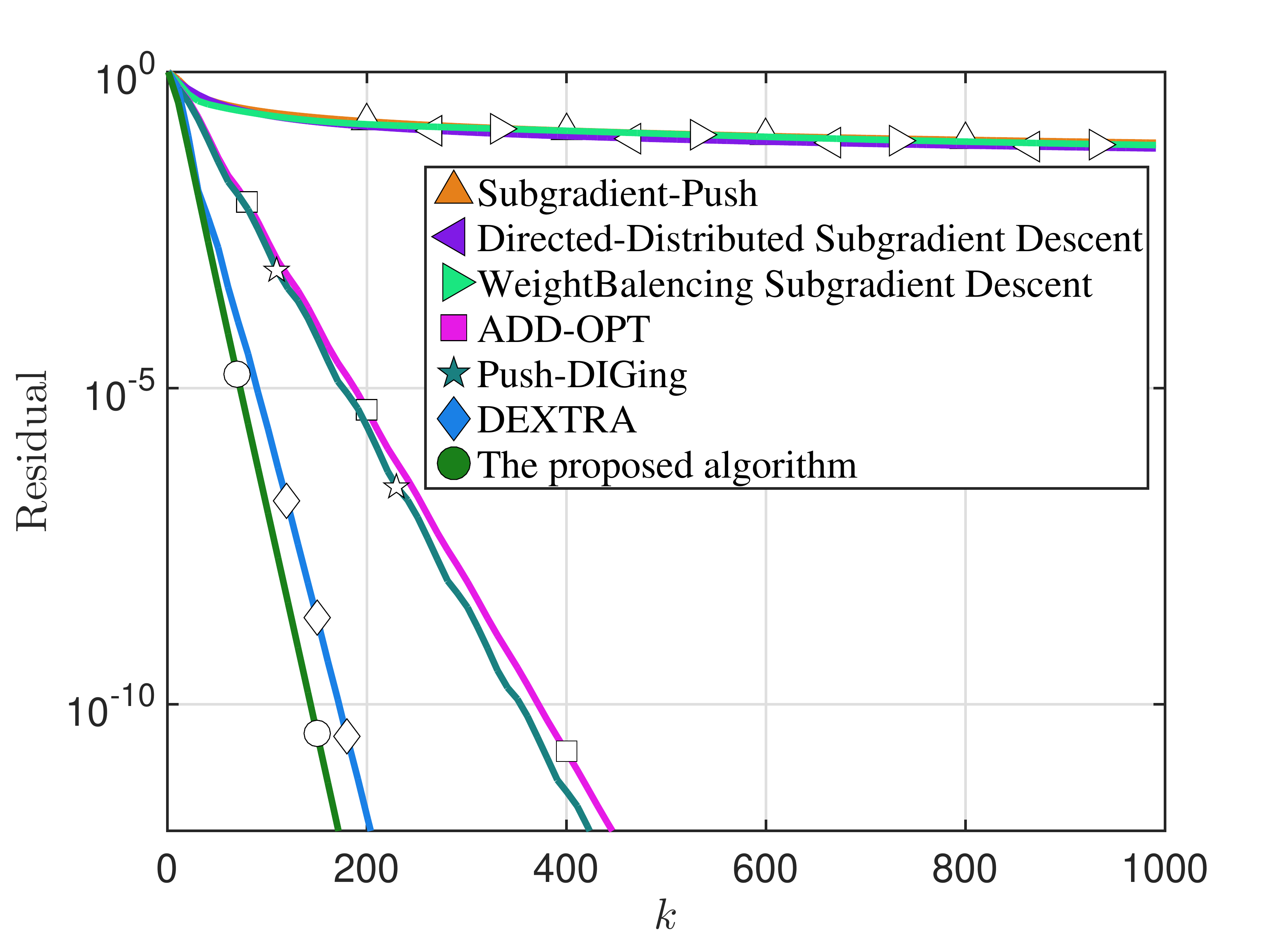}}
		\caption{Performance comparison between methods for directed networks.}\label{fig2}
	\end{center}
\end{figure}

\section{Conclusions}\label{s6}
This paper considers the distributed optimization problem of minimizing the sum of local objective functions over a multi-agent network, where the communication between agents is described by a \emph{directed} graph. Existing algorithms over directed graphs require at least the knowledge of neighbors' out-degree for all agents (due to the need for column-stochastic matrices). In contrast, our algorithm requires no such knowledge. Moreover, the proposed distributed algorithm achieves the best rate of convergence for this class of problems,~$O(\mu^k)$ for~$0<\mu<1$, given that the objective functions are strongly-convex with Lipschitz-continuous g, where~$k$ is the number of iterations.

{
	\footnotesize
	\bibliographystyle{IEEEbib}
	\bibliography{sample}
}

\end{document}